%% file: main.tex
\documentclass[11pt]{article}

\usepackage{underscore}
\usepackage{array}
\newcolumntype{P}[1]{>{\centering\arraybackslash}p{#1}}
\newcolumntype{M}[1]{>{\centering\arraybackslash}m{#1}}
\usepackage[applemac]{inputenc}
\usepackage[numbers,compress]{natbib}
\usepackage{amsthm}
\usepackage{amsmath}
\usepackage{algcompatible}
\usepackage{algorithm}
\usepackage{booktabs}
\usepackage{enumerate}
\usepackage{graphicx}
\usepackage{amssymb}
\usepackage{latexsym}
\usepackage{epstopdf}
\usepackage{color}
\usepackage{xcolor}
\usepackage{bbm}
\usepackage{needspace}
\usepackage{color, colortbl}
\usepackage[english]{babel}
\usepackage{tikz}
\usepackage{caption}
\usepackage{subcaption}
\usetikzlibrary{arrows,automata}
\usetikzlibrary{positioning}
\usepackage{filecontents}
\usepackage{mathtools}
\usepackage{algorithm}
\usepackage{multirow}
\usepackage{algpseudocode}
\usepackage{appendix} 
\usepackage{fullpage}   

\newcommand{\E}{\mathbb{E}}

\usepackage{arydshln}

\DeclareCaptionFont{mysize}{\fontsize{8.0}{9.6}\selectfont}
\theoremstyle{plain}
\newtheorem{thm}{\protect\theoremname}
\theoremstyle{plain}
\newtheorem{lem}[thm]{\protect\lemmaname}
\newtheorem{remark}{Remark}

\makeatother

\usepackage{babel}
\providecommand{\lemmaname}{Lemma}
\providecommand{\theoremname}{Theorem}

\input{macros}

\begin{document}


\title{Large-Scale System Identification  Using a Randomized SVD}
\author{Han Wang and James Anderson
\vspace{0.0625in}
\\
Columbia University}

\date{\today}

\maketitle

\pagenumbering{arabic}

\begin{abstract}
Learning a dynamical system from input/output data is a fundamental task in the control design pipeline. In the partially observed setting there are two components to identification: \emph{parameter estimation} to learn the Markov parameters, and \emph{system realization} to obtain a state space model. In both sub-problems it is implicitly assumed that standard numerical algorithms such as the singular value decomposition (SVD) can be easily and reliably computed. When trying to fit a high-dimensional model to data, for example in the cyber-physical system setting, even computing an SVD is intractable. In this work we show that an approximate matrix factorization obtained using randomized methods can replace the standard SVD in the  realization algorithm while maintaining the non-asymptotic (in  data-set size) performance and robustness guarantees of classical methods. Numerical examples illustrate that for large system models, this is the only method capable of producing a model.
\end{abstract}

\input{intro}
\input{formulation}
\input{RSVD}
\input{results}

\input{experiments}

\input{conclusion}

\bibliographystyle{IEEEtran}
\bibliography{reference}
\input{apps}
\end{document}

%% file: macros.tex
\newcommand{\hinf}{\mathcal{H}_\infty}

\DeclareMathOperator*{\rank}{rank}

\newcommand{\iid}{\stackrel{\mathclap{\text{\scriptsize{ \tiny i.i.d.}}}}{\sim}}
\newcommand{\R}{\mathbb{R}}
\newcommand{\C}{\mathbb{C}}

\newcommand{\Trace}{\text{Trace}}
\newcommand{\Toep}{\text{Toep}}
\newcommand{\ra}{\text{Range}}

\newcommand{\rsvd}{\mathtt{RSVD}}

\newcommand{\ja}[1]{\textcolor{blue}{#1}}

%% file: intro.tex
\section{Introduction}\label{sec:introduction}

System identification is the process of estimating parameters of a dynamical system from observed trajectories and input profiles. It is a fundamental component in the control design pipeline as many modern optimal and robust control synthesis methodologies rely on having access to a dynamical system model. Traditionally, system identification ~\cite{deistler1995consistency, peternell1996statistical, jansson1998consistency, bauer1999consistency, knudsen2001consistency}
was limited to asymptotic analysis, i.e., estimators were shown to be consistent under the assumption of infinite data. However, recent results~\cite{oymak2019non, tu2017non, sarkar2019finite, simchowitz2018learning, simchowitz2019learning} focus on the more challenging task of analyzing the finite sample setting. Theoretically, this type of analysis is more involved and requires tools from high dimensional probability and statistics.

In this paper, we consider the problem of identifying a discrete-time, linear time-invariant (LTI) system parameterized by the matrices $A\in \R^{n\times n}, B \in \R^{n\times m}, C\in \R^{p\times n}$ and $D\in \R^{p\times m}$, that evolves according to
\begin{equation}\label{eq:LTI}
\begin{aligned}
x_{t+1} &=A x_{t}+B u_{t}+ w_{t} \\
y_{t} &=C x_{t}+D u_{t}+v_{t},
\end{aligned}
\end{equation}
from $N< \infty$   observations of the output signal $\{y_t^{i}\}_{t=0}^{T}$ and control signal $\{u_t^{i}\}_{t=0}^T$ of length $T$ for $i=1, \cdots, N$. The vectors  $x_{t} \in \mathbb{R}^{n}, w_{t} \in \mathbb{R}^{n}$, and  $v_{t} \in \mathbb{R}^{p}$  in~\eqref{eq:LTI} denote the system state, process noise, and measurement noise at time $t$, respectively, the superscript denotes the output/input channel. In this setting, the problem is  referred to as being \emph{partially observed}.  A more simplistic setting occurs when one has access to (noisy) state measurements, i.e., $C=I, D=0,\ \text{and} \ v_t =0$ for all $t$. This is referred to as the \emph{fully observed} problem.

In the  fully observed setting, estimates for $(A,B)$ can be obtained by solving ordinary least-squares (OLS) optimization problems. A series of recent papers~\cite{abbasi2011online,simchowitz2018learning,sarkar2019finite,faradonbeh2018finite,dean2020sample} have derived non-asymptotic guarantees for ordinary least-squares (OLS) estimators. In the case of partially observed systems, which is conceptually more complicated than the fully observed case, OLS optimization can be used to estimate the \emph{Markov parameters} associated with~\eqref{eq:LTI} from which the Ho-Kalman algorithm~\cite{ho1966effective} can be employed to estimate the system parameters $(A,B,C,D)$. The process of obtaining estimates of the system matrices from the Markov parameters is referred to as \emph{system realization} which is the main focus of this paper. Using this framework, the authors of \cite{sarkar2019finite,oymak2019non,zheng2020non,lee2020improved,sun2020finite,tsiamis2019finite} have derived  non-asymptotic estimation error bounds for the system parameters which decay at a rate $O( \frac{1}{\sqrt{N}})$.  Note that these papers make different assumptions about the stability, system order, and the number of required trajectories to excite the unknown system. 

However, in contrast to the estimation error bounds, the computational complexity of system identification has received much less attention in the literature~\cite{sznaier2020control,reyhanian2021online}. Due to the fact that  the OLS problem is convex, and the computational bulk of the Ho-Kalman Algorithm is a singular value decomposition (SVD), it is taken for granted that system identification can be carried out at scale. As mentioned in~\cite{sznaier2020control}, with the increase of system dimension, the computational and storage costs of general control algorithms quickly become prohibitively large.  This challenge motivates us to design control algorithms that mitigate the ``curse of dimensionality''. In this paper, we aim to design an efficient and scalable system realization algorithm that can be deployed in the big data regime.

From the view of computational complexity, the system identification methods proposed in \cite{sarkar2019finite,oymak2019non,zheng2020non,lee2020improved,sun2020finite,tsiamis2019finite} are not scalable since the size of the Hankel matrix
increases quadratically with the length of output signal $T$ and cubically with the system state dimension $n$. The result is the  singular value decomposition used in  the  Ho-Kalman Algorithm cannot be computed. This quadratic/cubic dependence on the problem size greatly limits its application in large scale system identification problems. 

Motivated by the limits of the scalability  of numerical SVD computations,  there has been a surge of work which has focussed on providing approximate, but more easily computable matrix factorizations. Thanks to advances in our understanding of random matrix theory and high dimensional probability (in particular, concentration of measure), \emph{randomized methods} have been shown to provide an excellent balance between numerical implementation  (in terms of storage requirements and computational cost) and accuracy of approximation (in theory and practice). Broadly speaking this field is referred to as randomized numerical linear algebra (RNLA), and we refer the reader to ~\cite{voronin2015rsvdpack,halko2011finding, musco2015randomized} and the references therein for an overview of the field. In particular, the machine learning and optimization communities have started to adopt RNLA methods into their work flows with great success, see for example~\cite{PilW17,feng2018faster}.

The intuition is that randomized methods can produce efficient, unbiased approximations of nonrandom operations while being numerically efficient to implement by exploiting modern computational architectures such as parallelization and streaming. The RNLA framework we follow  (see~\cite{Woo14} for \emph{sketching}-based alternatives, or~\cite{KanV17} for random column sampling approaches) is a three step process ~\cite{halko2011finding}. First,  sample the range of the target matrix by carrying out a sequence of matrix-vector multiplications, where the vector is an ensemble of random variables. Next, an approximate low-dimensional basis for the range is computed. Finally, an exact matrix factorization in the low dimensional subspace is computed. The performance of the  randomized SVD (RSVD) has been studied in many works~\cite{voronin2015rsvdpack,halko2011finding, musco2015randomized,feng2018faster} and has found applications in large-scale problems across machine learning~\cite{yao2018accelerated}, statistics~\cite{drineas2016randnla}, and signal processing~\cite{oh2017fast}. 

The main contribution of this work is a stochastic Ho-Kalman Algorithm, where the standard SVD (which constitutes the  main computational bottleneck of the algorithm) is replaced with an RSVD algorithm, which  trades off accuracy and robustness for speed. We show that the stochastic Ho-Kalman Algorithm achieves the same robustness guarantees as its deterministic, non-asymptotic version in expectation. However, it outperforms the deterministic algorithm in terms of speed/computational complexity, which is measured by the total number of required floating-point operations (flops) [\cite{boyd2004convex}, \S C.1.1]. Compared with $ O(pmn^3)$ flops required by the deterministic algorithm, the stochastic Ho-Kalman Algorithm only requires $O(pmn^2 \log{n})$ flops. 

%% file: formulation.tex
\section{Preliminaries and Problem Formulation}
Given a matrix $A \in \C^{m\times n}$, where $\C$ is the set of complex numbers.  $\lVert A \rVert$ denotes the spectral norm and $\|A\|_F$ denotes the Frobenius norm, i.e., $\|A\|=\sigma_1(A)$, where $\sigma_1$ is the maximum singular value of $A$, and $\|A\|_F= \sqrt{\Trace{(A^*A)}}$, where $A^*$ denotes the Hermitian transpose of $A$. The multivariate normal distribution with
mean $\mu$ and covariance matrix $\Sigma$ is denoted by $\mathcal{N}\left(\mu, \Sigma \right).$ A matrix is said to be \emph{standard Gaussian} if every entry is drawn independently from  $\mathcal{N}\left(0, 1\right)$. Symbols marked with a tilde  are associated to the stochastic Ho-Kalman algorithm, while those with hats are associated with the deterministic Ho-Kalman algorithm.

\subsection{Singular Value Decomposition}
The singular value decomposition of the matrix $A \in \C^{m \times n}$, factors it as $A=U \Sigma V^{*}$, where $U\in \C^{m \times m}$ and $V\in \C^{n \times n}$ are orthonormal matrices, and $\Sigma$ is an $m \times n$ real diagonal matrix with entries $\sigma_{1}, \sigma_{2}, \cdots, \sigma_{n}$ ordered such that $\sigma_{1} \geq \sigma_{2} \geq \cdots \geq \sigma_{n} \geq 0$. When $A$ is real, so are $U$ and $V$. The truncated SVD of $A$ is given by $U_r \Sigma_r V^{\top}_r (r<\min\{m,n\})$, where the matrices $U_{r}$ and $V_{r}$ contain only the first $r$ columns of $U$ and $V$, and $\Sigma_{r}$ contains only the first $r$ singular values from $\Sigma$.  According to the Eckart-Young theorem~\cite{eckart1936approximation}, the best rank $r$ approximation to $A$ in the spectral norm or Frobenius norm is given by
\begin{equation}\label{eq:trunc}
A_{[r]}=\sum_{i=1}^{r} \sigma_{i} u_{i} v_{i}^{\top}
\end{equation}
where $u_{i}$ and $v_{i}$ denote the $i^{\text{th}}$ column of $U$ and $V$, respectively. More precisely, 
\begin{align}\label{eq:low_rank_opt}
\underset{\rank(X)\le r}{\text{minimize}}& ~\|A-X\| = \sigma_{r+1},
\end{align}
and a minimizer is given by $X^{\star}= A_{[r]}$. The expression~\eqref{eq:low_rank_opt} concisely sums up the scalability issue we are concerned with: on the left hand side is non-convex optimization problem with no polynomial-time solution; on the right is a singular value which for large $m$ and/or $n$ cannot be computed. In the sequel we shall  see how randomized methods can use approximate factorization to resolve these issues.

\subsection{System Identification}
We consider the problem of identifying a linear system model defined by Eq~\eqref{eq:LTI} where ${u}_{t} \iid \mathcal{N}\left(0, \sigma_{u}^{2} {I}_{m}\right), {w}_{t} \iid \mathcal{N}\left(0, \sigma_{w}^{2} {I}_{n}\right)$, and ${v}_{t} \iid \mathcal{N}\left(0, \sigma_{v}^{2} {I}_{p}\right).$ We further assume that the initial state variable $x_0=0_n$  (although the dimension, $n$, in unknown \emph{a priori}). Under these assumptions, we generate $N$ pairs of  trajectory, each of length $T$. We record the data as
$$
\mathcal D^N_T = \left\{\left(y_{t}^{i}, u_{t}^{i}\right): 1 \leq i \leq N, 0 \leq t \leq T-1\right\},
$$ where $i$ denotes $i^{\text{th}}$ trajectory and $t$ denotes $t^{\text{th}}$ time-step in each trajectory. With the data $\mathcal D^N_T$ the system identification problem can be solved in two steps:
\begin{enumerate}
    \item \textbf{Estimation:} Given $\mathcal D_T^N$,  estimate the first $T$ Markov parameters of the system which are defined as 
    \begin{equation*}
G=\left[D,~CB, ~CAB, ~ \ldots, ~C A^{T-2} B\right] \in \mathbb{R}^{m \times T p}. 
\end{equation*} 
Ideally, the estimation algorithm will produce finite sample bounds of the form $\|G-\hat G\|\le \epsilon(N,T)$. This is typically achieved by solving an OLS problem (see Appendix~\ref{sec:est}).
\item \textbf{Realization:} Given an estimated Markov parameter matrix $\hat G$, produce  state-space matrices $(\hat A, \hat B, \hat C,\hat D)$ with guarantees of the form $\|A-\hat A\| \le \epsilon_A, \|B-\hat B\| \le \epsilon_B$, etc. This is most commonly done using the Ho-Kalman algorithm. 
\end{enumerate}

In data collection, we refer the input/output trajectory
$(y_t, u_t), t=0,\cdots,T-1$ as  a \emph{rollout}. There are two approaches to collect the data. The first method involves multiple rollouts  \cite{tu2017non,sun2020finite,dean2020sample,zheng2020non}, where the system is run  and restarted with a new input signal $N$ times. The second approach  is the \emph{single rollout} method \cite{oymak2019non,simchowitz2018learning,simchowitz2019learning,sarkar2019finite,sarkar2019near}, where  an input signal is applied from time $0$ to $N \times T -1$ without restart. As this paper only focuses on the realization step of the problem, we can use either of aforementioned methods for data collection.

\subsection{System realization via noise-free Markov matrix $G$} 
With an estimated Markov matrix in hand, we wish to reconstruct the system parameters $A,B,C$ and $D$. To achieve this goal, we employ the  Ho-Kalman Algorithm \cite{ho1966effective}. We first consider the noise-free setting, i.e., $G$ is known exactly. The main idea of the Ho-Kalman Algorithm is to construct and factorize a Hankel matrix derived from the  $G$. Specifically, we generate a Hankel matrix as follows:
\begin{equation*}
\mathcal{H}=
\begin{bmatrix}
C B & C A B & \ldots & C A^{T_{2}} B \\ C A B & C A^{2} B  & \ldots & C A^{T_{2}+1} B \\ C A^{2} B & C A^{3} B  & \ldots & C A^{T_{2}+2} B \\ \vdots & \vdots & \vdots & \vdots \\ C A^{T_{1}-1} B & C A^{T_{1}}B  & \ldots & C A^{T_{1}+T_{2}-1} B\end{bmatrix} ,
\end{equation*}
where $T = T_1 + T_2 + 1.$ We use $\mathcal{H}^{-}$  $(\mathcal{H}^{+})$ to denote the $p T_1 \times m T_2$ Hankel matrix created by deleting the last (first) block column of $\mathcal{H}.$ We  assume that 
\begin{enumerate}
    \item the system~\eqref{eq:LTI} is observable and controllable, and
    \item $n = \text{rank}(\mathcal{H}) \le \min \{T_1, T_2\}.$
\end{enumerate}
Under these two assumptions, we can ensure that $\mathcal{H}$ and $\mathcal{H}^{-}$ are of rank $n$. We note that $\mathcal H^-$ can be factorized as
\begin{equation*}
\begin{aligned}
\mathcal{H}^{-}&=\begin{bmatrix}
C \\
C A \\
\vdots \\
C A^{T_{1}-1}
\end{bmatrix}
\begin{bmatrix}
B & A B & \ldots & A^{T_{2}-1} B
\end{bmatrix}\\
& := O Q,
\end{aligned}
\end{equation*}
where $O, Q$ denote the observability matrix and controllability matrix respectively.
We can also factorize $\mathcal{H}^{-}$ by computing its truncated SVD, i.e., $\mathcal{H}^{-}=U \Sigma_n V^T = (U {\Sigma_n}^{\frac{1}{2}})
({\Sigma_n}^{\frac{1}{2}} V^T).$ Therefore, the factorization of $\mathcal H^-$ establishes $O=U\Sigma_n^{\frac{1}{2}}, Q=\Sigma_n^{\frac{1}{2}}V^T.$ And doing so, we can obtain the system parameter $C$ by taking the first $p$ rows of $U {\Sigma_n}^{\frac{1}{2}}$ and the system parameter $B$ by taking the first $m$ columns of ${\Sigma_n}^{\frac{1}{2}} V.$ 
Then $A$ matrix can be obtained by $A=O^{\dagger} \mathcal{H}^{+} Q^{\dagger}$, where $(\cdot)^{\dagger}$ denotes the Moore-Penrose inverse.\footnote{The Moore-Penrose inverse of the matrix $A$ denoted by $A^\dagger$ is $V \Sigma^\dagger U^T$,  where $\Sigma^\dagger$ is formed by transposing $\Sigma$ and then taking the reciprocal of all the non-zero elements.
} Note that  $D$ is obtained without calculation since it is the first $p\times m$ submatrix of the Markov matrix $G$. It is worthwhile to mention that learning a state-space realization is a non-convex problem. There are multiple solutions yielding the same system input/output behavior and Markov matrix $G$. For example, if $(A,B,C,D)$ is a state-space realization obtained from  $G$, then $(SAS^{-1}, SB, CS^{-1},D)$ under a similarity transformation where $S$ is any  non-singular matrix is also a valid realization. 

\subsection{System realization via noisy Markov parameter $G$}
In the setting with noise, the same algorithm is applied to the estimated Markov matrix $\hat{G}$ (see Appendix~\ref{sec:est} for a simple method for obtaining $\hat G$ from $\mathcal D^N_T$), instead of the true matrix $G$. In this case the Ho-Kalman Algorithm will produce estimates $\hat{A},\hat{B},\hat{C}$ and $\hat{D}$.  The explicit algorithm is shown as Alg~\ref{alg:ho-kalman} (deterministic). It was shown in \cite{oymak2019non} that the robustness of the Ho-kalman Algorithm implies the estimation error for $A,B,C$ and $D$ is bounded by $O(\frac{1}{N^{1/4}})$, where $N$ is the sample size of trajectories, i.e.,
\begin{equation}
\begin{aligned}
\max &\left\{\lVert\hat{A}-S^{-1} A S\rVert, \lVert \hat{B}-S^{-1} B\rVert,\lVert \hat{C}-C S\rVert \right\} 
\leq  \hat{C}\sqrt{\lVert G-\hat{G}\rVert} =  O(\frac{1}{N^{1/4}})
\end{aligned}
\end{equation}
This result can be improved to $O(\frac{1}{\sqrt{N}})$ from~\cite{sarkar2019finite,tsiamis2019finite,lee2020improved}.

Note that the computational complexity of the Ho-Kalman Algorithm in Alg.~\ref{alg:ho-kalman} (deterministic) is dominated by the cost of computing the SVD (Step 7), which is $O(pT_1\times mT_2\times n)$ when using the Krylov method (see e.g. \cite{golub1996matrix,TrefB97}). Therefore, we want to use a small $T$ to reduce the computational cost. However, to satisfy the second assumption that $n = \text{rank}(\mathcal{H}) \le \min \{T_1, T_2\},$ where $T_1+T_2 +1 =T,$ the smallest $T$ we can choose is $2n+1$ with $T_1=T_2=n.$ In summary, the lowest achievable computational cost for SVD is $O(n^3).$ Such dependency on the system dimension is prohibitive for large-scale
 systems (e.g. systems with $n=100$ as we show in Section~\ref{sec:numerics}). Motivated by the drawbacks of the existing method,
we aim to answer the following question:
\begin{itemize}
    \item \textit{Is there a system realization method which can significantly reduce the computational complexity without sacrificing  robustness guarantees?}
\end{itemize}

The main  result of this paper is to answer this question in the affirmative. By leveraging randomized numerical linear algebra techniques described in the next section,  we design a stochastic version of the  Ho-Kalman  algorithm that is computationally efficient and produces competitive  robustness guarantees. An  informal version of our main results is given below:
\begin{thm}\textbf{(informal)}
The stochastic Ho-Kalman Algorithm reduces the computational complexity of the realization problem from $O(pmn^3)$ to $O(pmn^2 \log{n})$ when $T_1 = T_2 =\frac{n}{2}$. The achievable robustness is the same as deterministic Ho-Kalman Algorithm.

\end{thm}
In the next section we introduce the randomized methods and their theoretical and numerical properties. We then incorporate them into the Ho-Kalman algorithm and analyze its performance.

\begin{algorithm}
\caption{ Stochastic/Deterministic Ho-Kalman Algorithm} \label{alg:ho-kalman}
    \begin{algorithmic}[1]
    \State \textbf{Input:} Length $T$, Estimated Markov parameters $\hat{G}$, system order $n$, $(T_1, T_2)$ satisfying $T_1 + T_2 + 1 = T$ 
    \State \textbf{Outputs}: State space realization $\hat{A}, \hat{B}, \hat{C},\hat{D}$
      \State Generate a Hankel matrix $\hat{H} \in \mathbb{R}^{pT_1 \times m(T_1 + 1)}$ from $\hat{G}$
     \State $\hat{H}^{-} = \hat{H}(:,1:m T_2)$ \ja{ \Comment{$\text{dim}(\hat{H}^-)=p T_1\times m T_2}$}
       \State $\hat{H}^+ = \hat{H}(:,m+1:m (T_2+1))$ 
    \If{Deterministic}
    \State $\hat{L} = \hat{H}^-_{[n]}$  \Comment{\ja{truncated SVD via \eqref{eq:trunc} }}
     \State $\hat{U}, \hat{\Sigma}, \hat{V} = \mathtt{SVD}(\hat{L})$
     \ElsIf{Stochastic}
     \State $\hat{U}, \hat{\Sigma}, \hat{V} = \mathtt{RSVD}(\hat{H}^{-}, n,l)$\footnotemark 
      \ja{\Comment{$\tilde{L}=\hat{U} \hat{\Sigma} \hat{V} \approx\hat{L}}$}
      \EndIf
    \State $\hat{O}  = \hat{U}\hat{\Sigma}^{1/2}$ \ja{\Comment{$\text{dim}(\hat{O}) = pT_1 \times n$}}
     \State $\hat{Q} = \hat{\Sigma}^{1/2} \hat{V}^*$ \ja{\Comment{$\text{dim}(\hat{Q})=n \times m T_2}$}
    \State $\hat{C} = \hat{O}(1:p,:),\hat{B} = \hat{Q}(:,1:m)$
        \State  $\hat{A} =  \hat{O}^\dagger \hat{H}^+ \hat{Q}^\dagger, \hat{D} = \hat{G}(:,1:m)$
      \State \textbf{Return} $\hat{A} \in \mathbb{R}^{n \times n}, \hat{B} \in \mathbb{R}^{n \times m}, \hat{C} \in \mathbb{R}^{p \times n}, \hat{D} \in \mathbb{R}^{p \times m}$
       \end{algorithmic}
\end{algorithm}
\footnotetext{In the following analysis, we will use $\tilde{U}, \tilde{\Sigma}, \tilde{V},\tilde{L}$ to denote the variables used in the stochastic Ho-Kalman Algorithm.}

%% file: RSVD.tex
\section{Randomized singular value decomposition}
The  numerical computation  of a singular value decomposition can be implemented in many ways. The structure of the matrix to be decomposed will likely play a role in determining which is the most  efficient algorithm. We do not attempt to review methods here as  the  literature is vast. In the system realization problem, the Ho-Kalman Algorithm  computes the SVD of $\mathcal{H}^-$, a dense truncated block Hankel matrix. To the best of our knowledge there are no  specialized algorithms for this  purpose. As such, we assume we are dealing with a general dense low rank matrix.  A brief comparison between standard numerical methods and the randomized methods to be introduced next is given Section~\ref{sec:comp}.

The objective of the RSVD it to produce matrices $U, \Sigma, V$, such that for a given matrix $A\in \C^{m\times n}$ with $\rank(A)=r<\min\{m,n\}$, and tolerance $\epsilon>0$, the bound
\begin{equation*}
\|A-U \Sigma V^*\|\le \epsilon
\end{equation*}
is satisfied where $U$ and $V$ have orthornormal columns and $\Sigma\in \R^{k\times k}$ is diagonal with $k<r$. 

Following~\cite{halko2011finding}, the RSVD of a matrix $A$ with target rank $k$ is computed in two  stages (full  implementation  details are provided in Algorithm~\ref{alg:svd}):
\begin{enumerate}
\item Find a matrix $P \in \mathbb{R}^{m \times k}$ with orthonormal columns such that the range of $P$ captures as much of the range of $A$ as possible. In other words, $A \approx PP^*A$.
\item Form the matrix $M = P^*A \in \mathbb{R}^{k \times n}$ and apply the standard numerical linear algebra technique to compute the SVD of $P$.
\end{enumerate}
Step 1 is the \emph{range finding} problem. This is where randomization enters picture. Let $\omega^{(i)}$ be a standard  Gaussian vector, and compute  $y^{(i)}=A\omega^{(i)}$. This can be viewed as a sample of $\ra(A)$. Repeating this process $k$ times and concatenating samples into matrices we have $Y=A\Omega$, an orthonormal basis for $Y$ can then be computed using  standard techniques, we use an economy QR decomposition. Again, we concatenate basis vectors $q_i$ into a matrix $P$. Because $k$ is selected to be small, this process is computationally tractable. When $\|A-PP^*A\|$ is small, $PP^*A$ is a  good rank-$k$ approximation of $A$. In  step 2, standard deterministic routines are called to compute the SVD of $M$. These routines are considered tractable as the the matrix $M$ has dimension $k\times m$ where $k$ is ideally much less than $r$. From the SVD of $M$, the matrices $U, \Sigma, V$ can be easily constructed (lines 7--8 of $\rsvd$).

\begin{algorithm}
\caption{ Randomized SVD: $\mathtt{RSVD}$} \label{alg:svd}
    \begin{algorithmic}[1]
     \State\textbf {Input:} an $m \times n \ \text{matrix} \ A, \text {a target rank} \ k$,
     \newline
     $\text{an oversampling parameter} \ l$
     \State \textbf {Output:} Approximate SVD s.t. $A\approx USV^T$
     \State $\Omega=\mathtt{randn}(n, k+l)$
\State ${P}=\mathtt{orth}({A} \Omega)$ \ja{\Comment{approx. basis for $\ra{(A)}$}}
\State ${M}={P}^{\mathrm{T}} {A}$ \ja{\Comment{$\text{dim}(M) = (k+l) \times n$}}
\State $[{U}, {S}, {V}]=\mathtt{svd}({M})$
\State ${U}={P U}$
\State ${U}={U}(:, 1: k), {S}={S}(1: k, 1: k), {V}={V}(:, 1: k)$ 
\State \textbf{Return} ${U} \in \mathbb{R}^{m \times k}, {S} \in \mathbb{R}^{k \times k}, {V} \in \mathbb{R}^{n \times k} $
 \end{algorithmic}
\end{algorithm}

 In practice, if the target rank is selected to be $k$, then one should sample the range of $A$ $k+l$ times where $l$ is   a small interger. Typically $l=10$ is more than sufficient~\cite{halko2011finding}. In $\rsvd$, $\Omega$ is chosen to be a standard Gaussian matrix. Surprisingly, the computational bottleneck of $\rsvd$ is the  matrix-vector multiplication in computing $A\Omega$ in  step 4. To reduce the computational cost of this step, we can choose other types of random matrices such as the subsampled random Fourier transform (SRFT) matrix  which reduces the flop count from $O(mn(k+l)$ to $O(mn\log(k+l))$ without incurring much  loss in accuracy (we extend our results to this setting in Appendix~\ref{sec:SRFT}). It should be further noted that the computation of $A\Omega$ is trivially parallelizable.

The $\mathtt{orth}$ function called on  line  4 of $\rsvd$ computes an orthonormal basis for the range of its argument. This can be done in many ways, here we use an economy QR decomposition.

\subsection{Power Scheme for slowly decaying spectra} When the input matrix $A$ has a flat  spectrum, $\rsvd$ tends to struggle to find a good approximate basis. To improve the accuracy a power iteration scheme is employed~\cite[p.~332]{golub1996matrix}. Loosely, the  power iteration  are based on the observation that the singular vectors of $A$ and  $(AA^*)^qA$ are the same, while the singular values with magnitude less than  one will rapidly shrink.  In other words, it can reduce the effect of noise. More precisely, we apply $\rsvd$ to the matrix 
$ W =\left(A A^{*}\right)^{q} A $, and we have $$\sigma_j(W) = \sigma_j(A)^{2q+1},\quad  j=1,2,\cdots$$ which shows that for $\sigma_j <1,$ the power iteration will provide singular values that decay more rapidly while the singular vectors remain unchanged. This will provide a more accurate approximation, however it will require $2q+1$ times as many matrix vector multiplies.  The  following theorem provides a bound on the accuracy of the approximation that $\rsvd$ provides. We will make heavy use of this result in the sequel.
\begin{thm}\label{thm1}~\cite{halko2011finding}
 Suppose that $A$ is a real $m \times n$ matrix with singular values $\sigma_{1} \geq \sigma_{2} \geq \sigma_{3} \geq \cdots$. Choose a target rank $k \geq 2$ and an oversampling parameter $l \geq 2$, where $k+l \leq \min \{m, n\} .$ Then $\rsvd$ called on   $W=(AA^*)^q A$ with target rank $k$, and oversampling  parameter $l$ produces an orthonormal approximate basis $P$ which satisfies 
 \begin{equation*}
 \begin{aligned}
 \mathbb{E}\lVert(A-PP^*A)\rVert 
  \leq\left[1+\sqrt{\frac{k}{l-1}}+\frac{e \sqrt{k+l}}{l} \cdot \sqrt{\min \{m, n\}-k}\right]^\frac{1}{2q+1} \sigma_{k+1}
 \end{aligned}
 \end{equation*}
 \normalsize
 where $e$ is Euler's number and $\mathbb{E}$ denotes expectation with respect to the random  matrix $\Omega$.
 \end{thm}
Note that $\sigma_{k+1}$ is the theoretically optimal value in the deterministic setting. Thus the price of randomization is the given by the contents of the square brackets with exponent $\frac{1}{2q+1}$.

\textbf{Adaptive randomized range finder}. $\rsvd$ is implemented based on the assumption that the target rank $k$ is know \emph{a priori}. However, in practice, we may not know the true rank $k$ in advance. Therefore, it is desirable to design an algorithm that can find a matrix $P$ with as few orthonormal columns as possible such that 
$
\left\|\left({I}-P P^{*}\right) A\right\| \leq \varepsilon
$
where $\varepsilon$ denotes a given tolerance. The work in \cite{halko2011finding}, based on results from~\cite{WooLRT08} describes an  \emph{adaptive randomized ranger finder} that iteratively samples until the desired tolerance is obtained. It is worth noting that the CPU time requirements of $\rsvd$ and the adaptive version  are essentially identical. 

\subsection{Comparing RSVD with  SVD computaton}\label{sec:comp}
In this subsection, we briefly compare the cost of computing  a full singular value decomposition with that of computing an approximate decomposition. This  is a vast topic in numerical analysis and we refer the reader to ~\cite{TrefB97,golub1996matrix,li2019tutorial,halko2011finding,wall2003singular} for implementation and complexity details. We assume  that the  matrix we wish to factor is  dense with dimension $m\times n$. Broadly speaking, there are two approaches to computing an SVD: direct methods  based on a QR factorization, and iterative methods such as the  Krylov family  of algorithms. For computing an approximate  SVD, one can  use  the the truncated method~\eqref{eq:trunc} which involves computing a full SVD (which may not be  practical for large $m,n$), a rank-revealining QR decomposition  can be used, and finally randomized methods can be used.

In addition to the time complexity of the algorithm, it is also important  to consider how many times the matrix is read into memory. For large matrices, the cost of loading the matrix into memory outweighs  the computational cost. As such the number of passes over the matrix required by the algorithm needs to be  considered. We summarize the complexity of these algorithms in Table~\ref{tab:complexity}. The parameter $k$ indicates the target rank for approximate methods.

\begin{table}[htp]
    \centering
    \normalsize
    \begin{tabular}{llll}
    \hline
      Method   &  Time complexity  & Passes\\ \hline
      Full SVD ~\cite{trefethen1997numerical}& $O\left(mn \text{min}(m,n)\right)$ &$\min{(m,n)}$ \\ \hdashline
      Krylov~\cite{trefethen1997numerical}  & $O\left(mnk\right)$ & $k$ \\
      Truncated~ \cite{trefethen1997numerical} & $O\left(mn \text{min}(m,n)\right)$ &  $k$ \\ 
      RSVD ~\cite{halko2011finding}& $O\left(m n \log k\right)$  &  $1$ \\
      \hline
    \end{tabular}
    \caption{Comparison of SVD computation complexity. Methods below the dashed line are for computing an approximate SVD.}
    \label{tab:complexity}
\end{table}

From Table 1, we observe that RSVD is the fastest SVD algorithm among the aforementioned, especially when $m,n \gg k$. In term of space complexity, although they have the same dominant term, RSVD only requires one pass through the data while other  methods require multiple passes, which is prohibitively expensive for huge 
matrices. When the power iteration  is implemented in the $\rsvd$ algorithm, the  number of passes changes to $O(1)$.

%% file: results.tex
\section{Main Results}
The \emph{stochastic Ho-Kalman} algorithm we propose replaces the deterministic singular value decomposition and truncation (lines 7--8, Algorithm~\ref{alg:ho-kalman}) with a single approximate randomized SVD (line 10, Algorithm~\ref{alg:ho-kalman}) obtained using $\rsvd$. Proofs of all results are deferred to the appendix. For  the remainder of the paper, symbols with a tilde denote that they were obtained from the stochastic Ho-Kalman Algorithm, while symbols with a hat denote that they were obtained from the deterministic Ho-Kalman Algorithm. Finally, symbols with neither have been  obtained from the ground truth Markov matrix $G$.

From \cite{oymak2019non}, we have the following perturbation bounds for the deterministic Ho-Kalman Algorithm:
\begin{lem}\label{lemma1}~\cite{oymak2019non}
The matrices ${H}, \hat{{H}}$ and ${L}, \hat{{L}}$ satisfy the following perturbation bounds:
\begin{itemize}
\item $\max \Big \{ \lVert {H}^{+}-\hat{{H}}^{+}\rVert, \lVert {H}^{-}-\hat{{H}}^{-}\rVert \Big \} \leq \lVert {H}-\hat{{H}} \rVert
\leq \sqrt{\min \{T_{1}, T_{2}+1\}} \lVert {G}-\hat{{G}}\rVert .$
\item $ \lVert {L}-\hat{{L}}\rVert \leq 2\lVert{H}^{-}-\hat{{H}}^{-}\rVert \leq  2 \sqrt{\min \{ T_{1}, T_{2}\}} \lVert {G}-\hat{{G}}\rVert .$
\end{itemize}
\end{lem}

We will now use Theorem~\ref{thm1} and Lemma~\ref{lemma1} to provide average and deviation bounds on the performance of the stochastic Ho-Kalman algorithm. 
\begin{lem}\label{lem3} \textbf{(Average perturbation bound)} Denote $l\ge 2$ to be the oversampling parameter used in $\rsvd$. Run the Stochastic Ho-Kalman Algorithm with a standard Gaussian matrix $\Omega \in \mathbb{R}^{mT_2 \times(n+l)}$ in line 3 of $\rsvd$, where $n+l \le \min\{pT_1, mT_2\}.$ Then ${L}, \tilde{{L}}$ satisfy the following perturbation bound:
\begin{equation}\label{eq:eq2}\small
\begin{aligned}
\E\lVert L - \tilde{L}\rVert 
\le 2C_2\Big (2 + \sqrt{\frac{n}{l-1}} + \frac{e\sqrt{n+l}}{l} C_1  \Big) 
 \lVert G-\hat{G}\rVert 
\end{aligned}
\end{equation}
\normalsize
where 
\begin{equation*}
C_1= \sqrt{\min\{pT_1, mT_2\} -n},
\end{equation*}
 and
 \begin{equation*}
 C_2 =\sqrt{\min \{ T_{1}, T_{2}\}}.
\end{equation*}
Furthermore, if we exploit the power scheme with $\rsvd$, then the right-hand side of \eqref{eq:eq2} can be improved to
\begin{equation}\label{eq:eq3}\small
\begin{aligned}
4 C_2 \Big (1 + \frac{1}{2}\sqrt{\frac{n}{l-1}} + \frac{e\sqrt{n+l}}{2l} C_1  \Big)^{1/(2q+1)}   \lVert G-\hat{G}\rVert .
\end{aligned}
\end{equation}
\normalsize
\end{lem} 
\begin{proof}
See appendix.
\end{proof}
From \eqref{eq:eq2}, we have that the perturbation bound is determined by the ratio between the target rank $n$ and the oversampling parameter $l$. The error is large if $l$ is small. In practice, it is sufficient to use $l=5$ or $l=10$. And there is rarely any advantage to select $l>10$~\cite{halko2011finding}. In addition, from~\eqref{eq:eq3}, we know that the bound will decrease if we increase the power parameter $q$. The effect of $l$ in terms of running time and realization error is studied further in Section V where we observe that the stochastic Ho-Kalman algorithm is robust to the choice of $l$.
In case the average perturbation as characterized by Lemma~\ref{lem3} doesn't feel like a helpful quantity, a deterministic error bound  is also achievable: 

\begin{lem}\label{prob}\textbf{(Deviation bound)}
Let the assumptions of Lemma~\ref{lem3} hold. Assume $l \ge 4$ and let  $C_1$ and $C_2$ be defined  as in Lemma~\ref{lem3}. Then we have
\begin{equation}\label{eq:eqq4}\small
\begin{aligned}
  \lVert L - \tilde{L}\rVert 
    \le 2C_2 \Big (2 + 16\sqrt{1+\frac{n}{l-1}} + \frac{8\sqrt{n+l}}{l+1} C_1 \Big) 
 \lVert G-\hat{G}\rVert 
\end{aligned}
\end{equation}
with failure probability at most $3e^{-l}.$ Moreover, 
\begin{equation}\label{eq:eq5}\small
\begin{aligned}
  \lVert L - \tilde{L}\rVert 
    \le C_2\Big (2 + 6\sqrt{(n+l)l\log l} + 3\sqrt{n+l} \ C_1 \Big)
 \lVert G-\hat{G}\rVert 
\end{aligned}
\end{equation}
with failure probability at most $3l^{-l}.$
\end{lem}
\normalsize
\begin{proof}
See appendix.
\end{proof}

\begin{remark}
Another way to implement the stochastic Ho-Kalman Algorithm is to use a structured random matrix like subsampled random Fourier transform, or SRFT to compute the RSVD. In contrast with Gaussian matrix, SRFTs have faster matrix-vector multiply times. As a result $\rsvd$ computation time decreases.  We will present the bounds for SRFT matrix in the appendix.
\end{remark}


We are now ready to show the robustness of stochastic Ho-Kalman algorithm. The robustness result is valid up to a unitary transformation.

\begin{thm}\label{thm5}
Suppose the system ${A}, {B}, {C}, {D}$ is observable and controllable. Let ${O}, {Q}$ be order-n controllability/observability matrices associated with ${G}$ and $\tilde{{O}}, \tilde{{Q}}$ be approximate order-n controllability/observability matrices (computed by RSVD) associated with $\hat{{G}}$. Suppose $\sigma_{\min }({L})>0$ and the following robustness condition is satisfied:
$$
\E\lVert {L} - \tilde{{L}}\rVert \leq \sigma_{\min }({L}) / 2 .
$$
Then, there exists a unitary matrix ${S} \in \mathbb{R}^{n \times n}$ such that,
$$
\begin{array}{l}
\E\lVert{C}-\tilde{{C}} {S}\rVert_{F} \leq \E \lVert{O}-\tilde{{O}} {S}\rVert_{F} \leq \E\sqrt{5 n\lVert {L}-\tilde{{L}}\rVert}, \\
\E\lVert {B}-{S}^{*} \tilde{{B}}\rVert_{F} \leq \E \lVert{Q}-{S}^{*} \tilde{{Q}}\rVert_{F} \leq \E\sqrt{5 n\lVert{L}-\tilde{{L}}\rVert},
\end{array}
$$
and $\tilde{{A}}, {A}$ satisfy
\begin{equation*} \small 
\begin{aligned}
    \E\lVert A-{S}^{*} \tilde{{A}} {S}\rVert_{F} \le
    C_3  \left(\sqrt{\frac{\E \lVert L-\tilde{L}\rVert}{\sigma_{\min }({L})}} (\lVert {H}^{+}\rVert +\lVert{H}^{+}-\hat{{H}}^{+}\rVert )
    +\lVert {H}^{+}-\hat{{H}}^{+}\rVert \right),
\end{aligned}
\end{equation*}
where $C_3= \frac{14 \sqrt{n}}{\sigma_{\min }({L})}.$
\end{thm}

\begin{proof}
See appendix.
\end{proof}

As discussed in \cite{oymak2019non}, $\lVert{H}^{+}-\hat{{H}}^{+}\rVert,\E\lVert{L}-\hat{{L}}\rVert$ are perturbation terms that can be bounded in terms of  $\|{G}-\hat{{G}}\|$
via Lemma~\ref{lemma1} and Lemma~\ref{lem3}. Theorem~\ref{thm5} shows that the stochastic Ho-Kalman Algorithm has the same error bounds as its deterministic counterpart, which says the estimation errors for system matrix decrease as fast as $O(\frac{1}{N^{1/4}}).$ Our analysis framework can be easily extended to achieve the optimal error bounds $O(\frac{1}{\sqrt{N}})$ mentioned in~\cite{sarkar2019near,sarkar2019finite,tsiamis2019finite,lee2020improved}.

%% file: experiments.tex
\section{Numerical Experiments}\label{sec:numerics}
\subsection{Stochastic versus deterministic Ho-Kalman Algorithm}
We begin the comparison between the stochastic and deterministic Ho-Kalman Algorithm on six randomly generated systems described by~\eqref{eq:LTI}. For each system, its dimension $(n,m,p)$ is shown in the second column in Table~\ref{tab:1}. Each entry of the system matrix is generated through a uniform distribution over a range of integers as follows: matrix $A$ with random integers from $1$ to $5$, and matrices $B, C, D$ with random integers from $-2$ to $2$. The $A$ matrix is re-scaled to make it Schur stable\footnote{There is no requirement that the systems we work with be stable. However, we are using an $\hinf $-norm metric to judge the approximation error, so such an assumption makes things more straight forward.}, i.e., $\lvert \lambda_{\mathrm{max}}(A)\rvert <1$. The standard deviations of the process and measurement noises are $\sigma_w = 1$ and $\sigma_v= 0.5$. The length of trajectory $T$ is given in the second column in Table~\ref{tab:1} with $T_1$ chosen to be the smallest integer not less than $T/2$ and $T_2 = T-1 -T_1.$ The third column in Table~\ref{tab:1} denotes the matrix dimension of $\hat{H}^-$ when we run  Algorithm~\ref{alg:ho-kalman}. 

\begin{table*}[t]
\small
\centering
\begin{tabular}{|l|l|l|l|l|l|l|}\hline
	\multirow{2}{*}{Eg} & \multirow{2}{*}{$(n,m,p,T)$} & \multirow{2}{*}{dim$(\hat{H}^-)$} & \multicolumn{2}{l|}{\textbf{Running Time [s]}}&\multicolumn{2}{l|}{\textbf{Realization Error}} \\ \cline{4-7}
	& & & \textbf{deterministic} & \textbf{stochastic} & \textbf{deterministic}  & \textbf{stochastic} \\ \hline
	1 &(30,20,10,90) &$450\times880$
&0.1079&$0.0156$&7.64e-04&7.70e-04 \\ \hline
2&(40,30,20,100)&$2000\times2970$&5.7456&0.0897&6.67e-04&1.19e-03 \\ \hline
3 & (60,50,40,360)&$7200 \times 8950$& 227.0116&0.9323&8.27e-04&1.75e-03\\ \hline
4&(100,80,50,500)&$12500 \times 19920$&
922.8428&
4.4581&6.53e-04&1.66e-03\\ \hline
5&(120,110,90,600)&$27000\times 32890$
&Inf&17.6603& N/A&1.96e-03\\ \hline
6&(200,150,100,600)&$30000\times 44850$&Inf &52.1762& N/A&1.45e-03\\ \hline
\end{tabular}
\caption{Comparison between the stochastic and deterministic Ho-Kalman Algorithm. The running time is in seconds.  The approximate SVD is computed using $\rsvd$   with oversampling parameter $l=10$. To benchmark the algorithm performance,  a naive implementation of $\rsvd$ is used; we \emph{do not} use power iterations and \emph{do not}  make use of parallelization. Inf and N/A indicates that the deterministic algorithm fail to realize the system.}
\label{tab:1}
\end{table*}
\normalsize

We denote the true system as $\mathcal{G}(A, B, C,D)$ and the estimated system returned by the stochastic/deterministic Ho-Kalman algorithm as $\tilde{\mathcal{G}}(\tilde{A}, \tilde{B}, \tilde{C}, \tilde{D})/\hat{\mathcal{G}}(\hat{A}, \hat{B}, \hat{C}, \hat{D})$. We will use $\mathcal{G}, \tilde{\mathcal{G}}, \hat{\mathcal{G}}$ at times to reduce notational clutter. The realization error of the algorithm is measured by the 
normalized $\hinf$ error: $\frac{\lVert \tilde{\mathcal{G}}-\mathcal{G} \rVert_{\mathcal{G}_{\infty}}}{\lVert \mathcal{G} \rVert_{\mathcal{H}_{\infty}}}/\frac{\lVert \hat{\mathcal{G}}-\mathcal{G} \rVert_{\mathcal{H}_{\infty}}}{\lVert \mathcal{G} \rVert_{\mathcal{G}_{\infty}}}.$ The running time and the realization error of the deterministic and stochastic algorithms\footnote{We used the publicly available python package sklearn.utils.\allowbreak extmath.\allowbreak randomized_svd to compute the RSVD.} are reported in Table~\ref{tab:1} where the results for the stochastic algorithm are  average over 10 independent trials. All experiments are done on a $2.6$ GHz Intel Core i7 CPU.

The reported running time in the stochastic setting is highly conservative: we did not parallelize the sampling (i.e., constructing $A\Omega$ in line 4 of Algorithm~\ref{alg:svd}). Furthermore, as noted earlier (and further described in the appendix), standard Gaussian matrices are theoretically "nice" to work with but structured random matrices, such as SRFT matrices which compute $Y=A\Omega$ via a subsampled FFT~\cite{woolfe2008fast} will offer superior running times.

We observe that the stochastic Ho-Kalman algorithm consistently leads to a dramatic speed-up over the deterministic algorithm. The larger the system dimension is, the larger the run time gap is. It is worthwhile to mention that the deterministic Ho-Kalman Algorithm fails to provide a result in the $5^{\text{th}}$and $6^{\text{th}}$ examples where the system state dimensions are above 100. Meanwhile, the stochastic algorithm runs successfully and takes a fraction of the time the deterministic algorithm took to solve a 60 state realization problem. The stochastic algorithm can easily be applied to much larger systems. However with no means of comparison to existing algorithms, and having established the theoretical properties of the algorithm,  we do not pursue this avenue further here.

\subsection{Oversampling effects}
To illustrate the influence of oversampling (parameter $l$ in $\rsvd$), we run the stochastic Ho-Kalman Algorithm on the $4^{\text{th}}$ example $(n=100,m=80,p=50)$ in Table~\ref{tab:1} and vary the oversampling parameter $l$ from 1 to 10. In this experiment we use a power iteration parameter of $q=1$. Running times (averaged over 10 runs) and realization errors (averaged over 10 runs) are shown via a boxplot in Figure~\ref{fig:oversample1} and graph in Figure~\ref{fig:oversample2}. In the box-plot, the central red mark indicates the median, and the bottom and top edges of the box indicate the 25th and 75th percentiles, respectively. Outliers are denoted by "+". We observe that in Fig~\ref{fig:oversample2}, the realization error tends to be larger when a small oversampling number is used, although the change is slight. The observed behavior  is consistent with the theoretical analysis of Lemma~\ref{lem3}.  We can also observe from Fig~\ref{fig:oversample1} that the computational time is insensitive to the oversampling parameters, as such taking larger values of $l=10$ is advantageous.

\begin{figure*}
	\begin{subfigure}[t]{0.48\textwidth}
		\includegraphics[width=\textwidth]{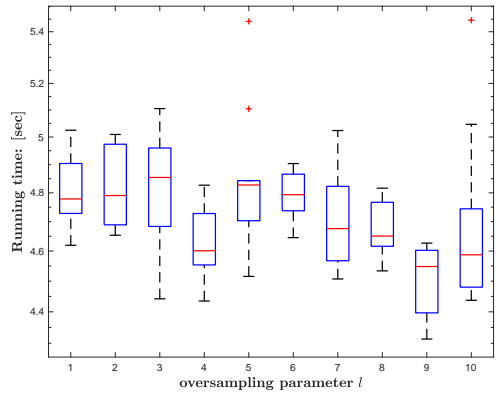}
     	\caption{Running time of stochastic Ho-Kalman Algorithm using $\rsvd$ with  oversampling parameter $l$. }
	    \label{fig:oversample1}
	\end{subfigure}
	\hfill 
	\begin{subfigure}[t]{0.48\textwidth}
		\includegraphics[width=\textwidth]{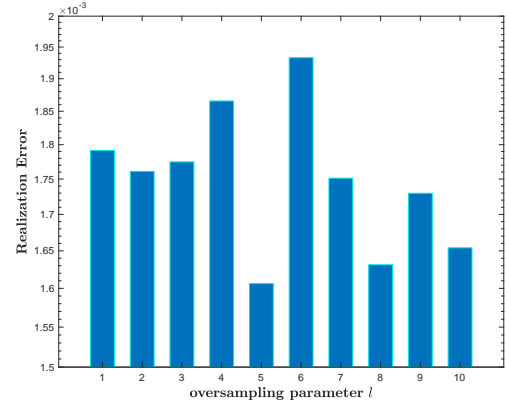}
		 \caption{Realization error of stochastic Ho-Kalman Algorithm using $\rsvd$ with  oversampling parameter $l$.}
	  \label{fig:oversample2}  
	\end{subfigure}
\vskip\baselineskip
\begin{subfigure}[t]{0.48\textwidth}
	\centering
	\includegraphics[width=\textwidth]{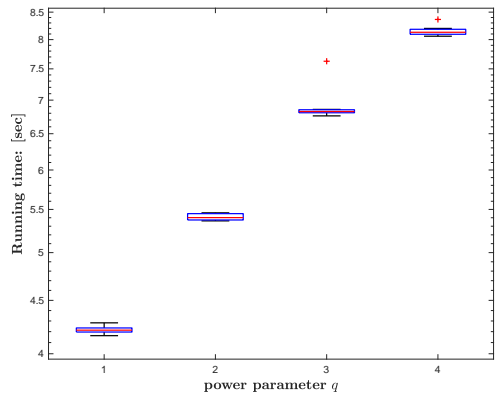}
	\caption{Running time of stochastic Ho-Kalman Algorithm with varying power parameter $q$. The oversampling parameter $l$ is 10.}
	\label{fig:power1}
\end{subfigure}
\hfill
\begin{subfigure}[t]{0.48\textwidth}
	\centering
	\includegraphics[width=\textwidth]{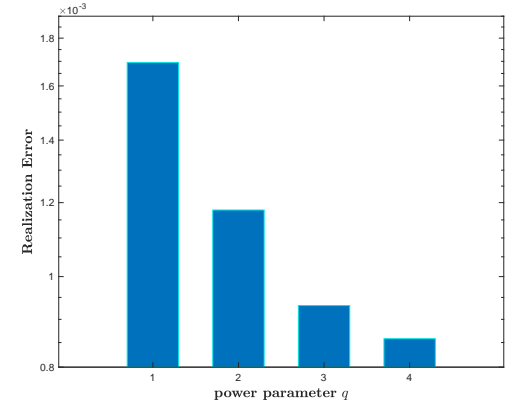}
	\caption{Realization error of the stochastic Ho-Kalman Algorithm with varying power parameter $q$. The oversampling parameter $l$ is 10.}
	\label{fig:power2}
\end{subfigure}
\caption{Oversampling and power iteration effect}
\end{figure*}

\subsection{Power iteration effect}
We now investigate numerically  how $\rsvd$ implemented with a power iteration impacts the performance of the stochastic Ho-Kalman Algorithm. Again we focus on example $4$. Based on the results of the previous subsection, we fix the oversampling parameter in $\rsvd$ as $l=10$, and sweep $q$ from 1 to 4. The results are shown   in Figures~\ref{fig:power1} and~\ref{fig:power2}. We observe that the realization error decreases as the power parameter $q$ increases as indicated in Eq~\eqref{eq:eq3}.  In contrast to the oversampling parameter $l$, the runtime demonstrably increases with $q$ at an empirically linear rate. This trend is expected and analyzed  in~\cite{halko2011finding}.

The power iteration method is most effective for problems where the spectrum of the matrix being approximated decays slowly. In the noise free setting,   $\rank(\mathcal H^-)=n$, where $n=100$ in this example. In contrast the dimensions of $\mathcal{H}^-$ are  $12500 \times 19920$. When noise is introduced, $\mathcal H^-$ becomes full rank and the spectral decay depends on $\sigma_w$ and $\sigma_v$. For the values chosen, these results show that spectral decay appears to be sharp enough that the power iterations do not offer \emph{significant} improvement in accuracy. However, as  $\sigma_v$ and $\sigma_w$ increase, the effect will become more dramatic.

%% file: conclusion.tex
\section{Conclusion}

We have introduced a scalable algorithm for system realization based on introducing randomized numerical linear algebra techniques into the Ho-Kalman algorithm. Theoretically it has been shown that our algorithm provides  non-asymptotic performance guarantees that are competitive with deterministic approaches. Furthermore, without any algorithm optimization, we have shown that the stochastic algorithm easily handles problem instances of a size significantly beyond what classical deterministic algorithms can handle.

In forthcoming work, using \emph{sketching}-based randomized methods, we have designed and analyzed a distributed  second-order algorithm for solving OLS problems for the Markov parameter estimation problem. We are currently working on the derivation of end-to-end performance bounds for the full randomized  system identification pipeline.

\section{Acknowledgements}
James Anderson and Han Wang acknowledge funding from the Columbia Data Science Institute. Han Wang is kindly supported by a Wei Family Foundation fellowship.

%% file: apps.tex
\begin{appendices} 
\section{Proofs} \label{appendix}
\subsection{Proof of Lemma~\ref{lem3} and~\ref{prob}}
 \begin{proof}[Proof of Lemma \ref{lem3}]
To prove \eqref{eq:eq2}, we first use the triangle inequality to get the following bound:
\begin{equation}\label{eq:first}
\E\lVert {L} - \tilde{{L}}\rVert
 \le \lVert {H}^{-} - \hat{{H}}^{-} \rVert + \E \lVert \hat{{H}}^{-} - \tilde{{L}}\rVert ,
\end{equation}
where $L = H^-$ and $H^-$ is of rank $n$. Then we bound $\E \lVert \hat{{H}}^{-} - \tilde{{L}}\rVert$ by applying Theorem~\ref{thm1} to $\hat{H}^{-}$ with $q=0$, giving:
\begin{align}
 \E \lVert \hat{{H}}^{-} - \tilde{{L}}\rVert &\le  
 \Big(1 + \sqrt{\frac{n}{l-1}} + \frac{e\sqrt{n+l}}{l} \sqrt{\min\{pT_1, mT_2\} -n}\Big )\lVert \hat{{H}}^{-} - \hat{{H}}^{-}_{[n]}\rVert \nonumber \\
 & \le \Big(1 + \sqrt{\frac{n}{l-1}} + \frac{e\sqrt{n+l}}{l} \sqrt{\min\{pT_1, mT_2\} -n}\Big )\lVert \hat{{H}}^{-} - H\rVert.
 \end{align}
The first inequality follows from $\tilde{{L}} = P P^*\hat{{H}}^{-}$. The second inequality is due to the fact that $\hat{{H}}^{-}_{[k]}$ is the best rank $k$ approximation of $\hat{{H}}^{-}$. Plugging the inequality above into~\eqref{eq:first} and applying Lemma~\ref{lemma1}, we obtain the inequality~\eqref{eq:eq2}.

Applying the bound in Theorem~\ref{thm1} to $\hat{H}^{-}$ with a fixed positive integer $q>0$ gives us
\begin{align}
 \E \lVert \hat{{H}}^{-} - \tilde{{L}}\rVert 
 & \le \Big(1 + \sqrt{\frac{n}{l-1}} + \frac{e\sqrt{n+l}}{l} \sqrt{\min\{pT_1, mT_2\} -n}\Big )^{1/(2q+1)} \nonumber \\
 &\quad \times \lVert \hat{{H}}^{-} - H\rVert + 1^{1/(2q+1)}\lVert \hat{{H}}^{-} - H\rVert \nonumber\\
 & \overset{(a)}{\leq} 2\Big(1+ \frac{1}{2}\sqrt{\frac{n}{l-1}} + \frac{1}{2}\frac{e\sqrt{n+l}}{l} \sqrt{\min\{pT_1, mT_2\} -n}\Big)^{1/(2q+1)}
  \lVert \hat{{H}}^{-} - H\rVert 
 \end{align}
Inequality (a) holds because $x^{1/(2q+1)}$ is concave in  $x$. Applying  Lemma~\ref{lemma1}, we prove the inequality~\eqref{eq:eq3}.
\end{proof}

 \begin{proof}[Proof of Lemma \ref{prob}]
We can follow the same steps as in the proof of Lemma~\ref{lem3} and use the bound given in Corollary 10.9~\cite{halko2011finding}.
\end{proof}

\subsection{Perturbation bounds for Stochastic Ho-Kalman Algorithm with SRFT Test Matrices}\label{sec:SRFT}
The subsampled random Fourier transform (SRFT), which might be the simplest structured random matrix, is an $n \times \ell$ matrix of the form
$$
\Omega=\sqrt{\frac{n}{\ell}} D F R
$$
where $D$ is an $n \times n$ diagonal matrix whose entries are independent random variables uniformly distributed on the complex unit circle, $F$ is the $n \times n$ unitary discrete Fourier transform (DFT), whose entries take the values 
\begin{equation*}
f_{p q}=n^{-1 / 2} \mathrm{e}^{-2 \pi i(p-1)(q-1) / n} ~\text{for}~ p, q=1,2, \ldots, n,
\end{equation*}
and $R$ is an $n \times \ell$ matrix that samples $\ell$ coordinates from $n$ uniformly at random~\cite{halko2011finding}. When $\Omega$ is a SRFT matrix, we can calculate the matrix multiplication $Y=A \Omega$ using $\mathrm{O}(m n \log (\ell))$ flops by applying a subsampled FFT~\cite{woolfe2008fast}. 
\begin{lem}\label{lemmm} \textbf{(Deviation bound)} Denote $l\ge 2$ to be the oversampling parameter used in RSVD algorithm. Run the Stochastic Ho-Kalman Algorithm with a SRFT matrix $\Omega \in \mathcal{R}^{mT_2 \times(n+l)}$ in computing the RSVD step, where $4[\sqrt{n}+\sqrt{8 \log (n mT_2)}]^{2} \log (n)  \leq l +n\leq mT_2 $. Then ${L}, \tilde{{L}}$ satisfy the following perturbation bound:
\begin{equation}\label{eq:eq4}\small
\begin{aligned}
  \lVert L - \tilde{L}\rVert 
    \le (1+\sqrt{1+ \frac{7mT_2}{l+n}} )\times 2 \sqrt{\min \{ T_{1}, T_{2}\}} \lVert G-\hat{G}\rVert 
\end{aligned}
\end{equation}
with failure probability at most $\mathrm{O}\left(n^{-1}\right)$.
\end{lem}
\begin{proof}
We can follow the same steps as in the proof of Lemma~\ref{lem3} and use the bound given in Theorem 11.2 of~\cite{halko2011finding} to finish the proof.
\end{proof}
\subsection{Proof of Theorem~\ref{thm5}}
To prove Theorem~\ref{thm5}, we require two auxiliary lemmas.
\begin{lem}\label{lemma4}
Suppose $\sigma_{\min }({L}) \geq 2 \E\lVert {L} - \tilde{{L}}\rVert$ where $\sigma_{\min }({L})$ is the smallest nonzero singular value (i.e. $n$-th largest singular value) of ${L}$. Let rank $n$ matrices ${L}, \tilde{{L}}$ have the singular value decomposition ${U} {\Sigma} {V}^{*}$ and $\tilde{{U}} \tilde{\Sigma} \tilde{{V}}^{*}$. There exists an $n \times n$ unitary matrix ${S}$ so that
\begin{equation}\label{eq:LHS}
\begin{aligned}
\E \lVert {U} {\Sigma}^{1 / 2}-\tilde{{U}} \tilde{\Sigma}^{1 / 2} {S}\rVert_{F}^{2}+\E\lVert V {\Sigma}^{1 / 2}-\tilde{V} \tilde{\Sigma}^{1/2}S \rVert_{F}^{2} \leq 5 n \E\lVert {L} - \tilde{{L}}\rVert .
\end{aligned}
\end{equation}
\end{lem}
\begin{proof}: Direct application of Theorem $5.14$ of~\cite{tu2016low} guarantees the existence of a unitary ${S}$ such that
\begin{equation}
\begin{aligned}
\mathrm{LHS}=&\E\lVert {U} {\Sigma}^{1 / 2}-\tilde{{U}} \tilde{\Sigma}^{1 / 2} {S}\rVert_{F}^{2}+\E\lVert V {\Sigma}^{1 / 2}-\tilde{{V}} \tilde{\Sigma}^{1 / 2} {S}\rVert_{F}^{2} \\
&\leq \frac{2}{\sqrt{2}-1} \frac{\E\lVert{L}-\tilde{{L}}\rVert_{F}^{2}}{\sigma_{\min }({L})},
\end{aligned}
\end{equation}
where LHS refers to the left hand side of~\eqref{eq:LHS}. To proceed, using $\E\operatorname{rank}({L}-\tilde{{L}}) \leq 2 n$ and by assumption $\sigma_{\min }({L}) \geq 2\E\|{L}-\tilde{{L}}\| \geq \sqrt{2 / n}\E\|{L}-\tilde{{L}}\|_{F}$, we find
$
\text { LHS } \leq \frac{\sqrt{2 n}}{\sqrt{2}-1}\E\|{L}-\tilde{{L}}\|_{F} \leq \frac{2 n}{\sqrt{2}-1}\E\|{L}-\tilde{{L}}\| \leq 5 n\E\|{L}-\tilde{{L}}\|$ .
\end{proof}

\begin{lem}
 Suppose $\sigma_{\min }({L}) \geq 2\geq 2 \E\lVert {L} - \tilde{{L}}\rVert .$ Then, $\E\|\tilde{{L}}\| \leq 2\|{L}\|$ and $\sigma_{\min }(\E\tilde{{L}}) \geq \sigma_{\min }({L}) / 2$.
 \end{lem}
 \begin{proof}
 See Lemma 2.2 in~\cite{oymak2019non}.
 \end{proof}
 Using these, we will prove the robustness of the stochastic Ho-Kalman Algorithm, which is stated in Theorem~\ref{thm5}. The robustness will be up to a unitary transformation similar to Lemma~\ref{lemma4}.
 \begin{proof}
 The proof is obtained by following the proof of Theorem 5.3 in ~\cite{oymak2019non} and  substitute $\tilde {L}$ for $\hat {L}$.
 \end{proof}

\section{Markov parameters estimation by least squares} \label{sec:est}
 Given a sequence $\{z_i\}_{i=0}^{k-1}$, the operator $\Toep(z)$ returns a $k\times k$ upper-triangular Toeplitz matrix $Z$, where
\begin{equation*}
Z_{i,j} = Z_{i+1,j+1}=z_{i-j}, \quad \text{if } i\le j,
\end{equation*}
and $Z_{i,j}=0$ when $i>j$.

We will briefly introduce some existing results on learning the Markov parameter matrix $G$. The matrix $G$ can be learned by solving the following OLS problem:
\begin{equation}\label{eq:leastsq}
\hat{G}=\underset{{X} \in R^{p \times mT}}{\operatorname{argmin}}\|{Y}-{X} {U}\|_{F}^{2}={Y} {U}^{T}\left({U} {U}^{T}\right)^{-1}
\end{equation}
where
$$
{Y} =\left[{y}^{(1)} \ldots {y}^{(N)}\right]\in \mathbb{R}^{p\times NT},
$$
$${U} =\left[{U}^{(1)} \ldots {U}^{(N)}\right] \in \mathbb{R}^{mT \times NT}.
$$
Where ${U}^{(i)}=\Toep(u_0^{(i)},\hdots, u_{T-1}^{(i)})$ and $y^{(i)}=[y_{0}^{(i)} \ y_{1}^{(i)} \ \ldots \ y_{T-1}^{(i)}] \in \mathbb{R}^{p \times T}$. Note that there are different methods from Eq~\ref{eq:leastsq} to formulate the least square problem, depending on how the input/output data samples are collected and utilized~\cite{oymak2019non,zheng2020non}. All the existing work \cite{sun2020finite,oymak2019non,sarkar2019finite,simchowitz2019learning,tu2017non,zheng2020non} shows that the estimated Markov parameters converge to the true Markov parameters at a rate of $O(\frac{1}{\sqrt{N}})$, where $N$ is the number of trajectories, regardless of the algorithm used.
 
\end{appendices}